\documentclass{amsart}

\usepackage{upgreek}
\usepackage{graphicx}
\usepackage{color}
\usepackage{subfigure}
\newtheorem{theorem}{Theorem}[section]
\newtheorem{prop}{Proposition}[section]
\newtheorem{coro}{Corollary}[section]
\newtheorem{lemma}[theorem]{Lemma}

\theoremstyle{definition}
\newtheorem{definition}[theorem]{Definition}

\theoremstyle{remark}
\newtheorem{remark}[theorem]{Remark}

\numberwithin{equation}{section}

\newcommand{\abs}[1]{\lvert#1\rvert}
\newcommand{\norm}[1]{\lVert#1\rVert}

\DeclareMathOperator{\im}{Im}
\newcommand{\ud}{\mathrm{d}}

\begin{document}

\title[Continuity of the Lyapunov Exponent]{Continuity of the Lyapunov Exponent for analytic quasi-perodic cocycles with singularities}

\author{S. Jitomirskaya and C. A. Marx}
\address{Department of Mathematics, University of California, Irvine CA, 92717}
\thanks{The work was supported by NSF Grant DMS - 0601081 and BSF, grant 2006483. .}





\begin{abstract}
We prove that the Lyapunov exponent of quasi-periodic cocyles with singularities behaves continuously over the analytic category. We thereby generalize earlier results, where singularities were either excluded completely or constrained by additional hypotheses. Applications are one-parameter families of analytic Jacobi operators, such as extended Harper's model describing crystals subject to external magnetic fields.
\end{abstract}

\maketitle
\centerline{ \small \it {Dedicated to Richard S. Palais on the occasion of
     his 85th birthday.}}

\section{Introduction} \label{sec_intro}

Denote by $\mathbb{T}:=\mathbb{R} / \mathbb{Z}$ the torus equipped with its Haar measure $\mu$, $\mu(\mathbb{T}) = 1$. Given $\beta$ irrational and a measurable $D: \mathbb{T} \to M_{2}(\mathbb{C})$ satisfying $\log{\abs{\det{D}}} \in L^1(\mathbb{T}, \ud \mu)$, a {\em{cocycle}} is a pair $(\beta,D(x))$, understood as linear skew-product acting on  $\mathbb{T} \times \mathbb{C}^{2}$ by $(x,v) \mapsto \left(x+\beta, D(x) v\right)$.

Using the sub-additive ergodic theorem, for any cocycle $(\beta, D)$ one can define the Lyapunov-exponent (LE) by
\begin{equation} \label{eq_LE}
L(\beta, D) = \lim_{n \to \infty} \frac{1}{n} \log \norm{D(x + (n-1) \beta) \dots D(x)} = \lim_{n \to \infty} \frac{1}{n} \int \norm{D(x)} \ud \mu(x) ~\mbox{.}
\end{equation}

In this paper we would like to analyze the dependence of the LE on the matrix valued function $D$ upon variation over the analytic category.  In view of the following definition, given  a Banach space $X$ and $\delta>0$, we denote by $\mathcal{C}_\delta^\omega(\mathbb{T}, X)$ the analytic $X$-valued functions on $\mathbb{T}$ with extension to a neighborhood of $\mathbb{T}_\delta:=\{\abs{\im{z}} \leq \delta\}$ (``the $\delta$-strip of $\mathbb{T}$''). 

\begin{definition} \label{def_anacoc}
Let  $(\beta,D(x))$ be a cocyle. If $D \in \mathcal{C}_\delta^\omega(\mathbb{T}, M_2(\mathbb{C}))$ for some $\delta>0$, we call $(\beta,D(x))$ an {\em{analytic cocyle}}. 
An analytic cocycle $(\beta,D(x))$ is called {\em{singular}} if $\det (D(x_0)) = 0$ for some $x_0 \in \mathbb{T}$, in which case $x_0$ is referred to as singularity of the cocycle $(\beta,D(x))$.
\end{definition}
We amend that analyticity automatically guarantees $\left \vert \int \log{\abs{\det{D}(x)}} \ud \mu(x) \right \vert < \infty$ (for a simple argument see the proof of Lemma \ref{eq_lemint}). 

For $\delta > 0$, let $D \in \mathcal{C}_\delta^\omega(\mathbb{T},M_2(\mathbb{C}))$.  Setting $\det{D(x)} =: d(x)$, for $d(x)$ not vanishing identically, analyticity allows for only finitely many zeros. In particular, for our analysis it will then prove useful to define the renormalization $D^\prime(x)$,
\begin{equation}
D^\prime(x) := \frac{1}{\sqrt{\abs{d(x)}}} D(x) ~\mbox{.}
\end{equation}
To simplify notation, we write 
\begin{equation} \label{prime}
L^{\prime}(\beta, D):=L(\beta, D^{\prime})
\end{equation}

Finally, the space $\mathcal{C}_\delta^\omega(\mathbb{T},M_2(\mathbb{C}))$ is naturally equipped with a topolgy induced by the norm $\vert \vert D(z) \vert \vert_\delta := \sup_{\abs{\im{z}} \leq \delta} \norm{D(z)}$, where $\norm{.}$ denotes the usual matrix norm. As our main result, we establish the following:
\begin{theorem} \label{thm_cont_new}
Let $\delta>0$. For fixed Diophantine $\beta$, the Lyapunov exponents $L(\beta, .)$ and $L^\prime(\beta,.)$ are continuous on $\mathcal{C}_\delta^\omega(\mathbb{T},M_2(\mathbb{C}))$ with respect to the topology induced by $\vert \vert . \vert \vert_\delta$. 
\end{theorem} 
We recall that $\beta$ is called Diophantine, if there exists $0 < b(\beta)$ and $1 < r(\beta)< + \infty$ such that for all $j \in \mathbb{Z}\setminus \{0\}$
\begin{equation} \label{eq_diophantine}
\abs{\sin(2\pi j \beta)} > \dfrac{b(\beta)}{\abs{j}^{r(\beta)}} ~\mbox{.}
\end{equation}

Besides from being a natural question to ask, our motivation for Theorem \ref{thm_cont_new} comes from the spectral theory of quasi-periodic analytic Jacobi operators on $\mathit{l}^2(\mathbb{Z})$,
\begin{eqnarray} \label{eq_hamiltonian}
& (H_{\theta;\beta} \psi)_k := v(\theta + \beta k) \psi_{k} + c(\theta + \beta k) \psi_{k+1} + \overline{c}(\theta + \beta (k-1)) \psi_{k-1} ~\mbox{.}
\end{eqnarray}
Here, $\beta$ is a fixed irrational and $v$, $c$ are $\mathcal{C}_\delta^\omega(\mathbb{T}, \mathbb{C})$ for some $\delta>0$. Moreover, $v$ is taken to be a real-valued function on $\mathbb{T}$, which makes (\ref{eq_hamiltonian}) a bounded self-adjoint operator for each $\theta \in \mathbb{T}$. An important special case of (\ref{eq_hamiltonian}) is given by $c(x) = 1$ (Schr\"odinger operators). 

Spectral analysis of (\ref{eq_hamiltonian}) amounts to the study of solutions to the finite difference equation $H_{\theta;\beta} \psi = E \psi$ over $\mathbb{C}^\mathbb{Z}$. It is well known that this problem can be tackled from a dynamical point of view, defining the transfer matrix,
\begin{equation} \label{eq_transfer}
B^E(x) := \dfrac{1}{c(x)} \begin{pmatrix} E - v(x) & -\overline{c}(x - \beta) \\ c(x) & 0 \end{pmatrix}  ~\mbox{.}
\end{equation}
We are particularly interested in operators (\ref{eq_hamiltonian}) where $c(x)$ is not bounded away from zero. In this case, the transfer matrix $B^E(x)$ is well defined except for finitely many points determined by the zeros of $c(x)$. 

The dynamical system relevant to the spectral analysis of (\ref{eq_hamiltonian}) is then given by the cocyle $(\beta, B^E)$. In addition, there is also an associated analytic cocyle $(\beta, A^E)$ given by 
\begin{equation}
A^E(x) := \begin{pmatrix} E - v(x) & -\overline{c}(x-\beta) \\ c(x) & 0 \end{pmatrix} ~\mbox{.}
\end{equation}
Obviously, for $\mu$ a.e. $x$ the two relevant cocyles  are related by $B^E(x) = \dfrac{1}{c(x)} A^E(x)$. In particular, for a given Jacobi operator, we obtain the following relation between the LE of its associated cocyles $(\beta, B^E)$ and $(\beta, A^E)$,
\begin{equation}
L(\beta, B^E) = L^\prime(\beta, A^E) ~\mbox{.}
\end{equation}

Oftentimes one studies one-paramter families of quasi-periodic Jacobi matrices. For instance, a two dimensional crystal layer subject to an external magnetic field of flux $\beta$ perpendicular to the lattice plane may be described by an operator of the form (\ref{eq_hamiltonian}) with functions $c$, $v$ given by
\begin{equation} \label{eq_hamiltonian1}
c(x) := \lambda_{3} \mathrm{e}^{-2\pi i (x+\frac{\beta}{2})} + \lambda_{2} + \lambda_{1} \mathrm{e}^{2 \pi i (x+\frac{\beta}{2})} ~\mbox{,}
~ v(x)  := 2 \cos(2 \pi x) ~\mbox{.}
\end{equation}  
Here, the parameter $\lambda = (\lambda_1, \lambda_2, \lambda_3)$ models the lattice geometry as well as the interactions between the nuclei situated at the lattice points of $\mathbb{Z}^2$. The operator associated with (\ref{eq_hamiltonian1}) is known as extended Harper's operator \cite{E,G,T,U}. We mention that a prominent special case of (\ref{eq_hamiltonian1}) arises for $\lambda_1 = \lambda_3 = 0$, the associated operator being known as almost Mathieu operator (or Harper's operator in physics literature). 

For such one-parameter families, it is a natural conjecture to expect continuity of the LE upon variation of the parameter $\lambda$. For extended Harper's model this question constitutes an important ingredient for the spectral analysis, which so far is only known for the almost Mathieu case. Theorem \ref{thm_cont_new}  answers this question from a general point of view. 
 
Continuity of the Lyapunov exponent for analytic cocycles has been the subject of earlier studies. These considerations however imposed restrictions on the determinant $d(x)$. 

In \cite{S}, H\"older continuity of the LE was established for Schr\"odinger cocyles (thus c=d=1) under a strong Diophantine condition with $\abs{j}^{r(\beta)}$ in (\ref{eq_diophantine}) replaced by $\abs{j} \log{\abs{j}}^{r(\beta)}$. 

An analogue of Theorem \ref{thm_cont_new} for $d(x)=1$ (or more generally for $d(x)$ bounded away from zero) was proven in  \cite{C}. This statement in particular implies continuous dependence of the LE on the coupling constant for the almost Mathieu operator. 

Later, in \cite{A}, $d(x)$ was allowed to vanish, however, in order to deal with these zeros, the space of interest was restricted to analytic cocycles having the {\em{same}} $d(x)$ (see Theorem 1 in \cite{A}). Applied to extended Harper's equation, the latter result already implied continuity of $L(\beta, A^E)$ in the energy, however, since variations of $\lambda$ change $\det A^E$, it does not yield continuity in the coupling \footnote{We correct a false remark in \cite{A} which asserted joint continuity of the Lyapunov exponent for extended Harper's equation based on Theorem 1 therein. In \cite{E} (based on \cite{A}), the authors, however, only used continuity with respect to $E$ which, as mentioned, does indeed follow from Theorem 1 proven in \cite{A}.}.

At this point we mention that when allowing $d(x)$ to vanish, details of the number theoretic nature of the frequency $\beta$ come into play. Whereas the earlier result in \cite{C} is valid for any irrational $\beta$, the result in \cite{A} could only be proven for Diophantine $\beta$. 

The achievement here is to deal with {\em{all}} non-trivial singular analytic cocycles $(\beta,D)$, thus removing the above mentioned constraints on $d(x)$ imposed by previous studies. Allowing for zeros in $d(x)$ however, results in signatures of the arithmetic properties of $\beta$ which manifest themselves in a Diophantine condition on $\beta$. 

Following, we employ a similar general strategy as in \cite{A} to prove Theorem \ref{thm_cont_new}. However, allowing the determinant to vary requires changes in the heart of the proof of \cite{A} where the authors provide a large deviation bound for analytic non-SL(2,$\mathbb{C}$) cocycles (Lemma 1 in \cite{A}). 

The key to Theorem \ref{thm_cont_new} is to appropriately generalize this large deviation bound to also incorporate a variation of $d(x)$. With this new, {\em{uniform}}, large deviation bound at hand, the remainder of the proof given in \cite{A} carries over more or less literally to imply Theorem \ref{thm_cont_new}. 

The paper is organized as follows. As preparation, in Sec. \ref{sec_unifloj} we establish a uniform version of the Lojasiewicz inequality, Theorem  \ref{thm_opentrans}, allowing us to deal with zeros of $d(x)$ upon continuous variation of the cocycle. This enters as crucial ingredient in the proof of the uniform large deviation bound, Theorem \ref{thm_LDB}, established in Sec. \ref{sec_unifdevbd}. Finally, we conclude with some corollaries to the main theorem in Sec. \ref{subsec_cont_rem}.

\section{Openness of $\alpha$-transversality} \label{sec_unifloj}

We start with some preparations exploring basic properties of complex analytic functions. To deal with possible zeros of $d(x)$, in \cite{A} the authors made use of the following basic fact valid for every {\em{real}} analytic function $f(x)$:
\begin{theorem}[Lojasiewicz inequality \cite{P}]
Given a real analytic function $f(x)$ on $\mathbb{T}$, there exist constants $0 < \alpha, \epsilon_0 \leq 1$ such that
\begin{equation} \label{eq_loj_orig}
\mu\{ \abs{f(x)} < \epsilon \} < \epsilon^\alpha ~\mbox{,}
\end{equation}
for every $0< \epsilon < \epsilon_0$.
\end{theorem}

Note that the exponent $\alpha$ as well as $\epsilon_0$ depend on the function $f$. For our purposes we would like to be able to choose these constants {\em{uniformly}} over functions sufficiently close in a suitable topology. 

Since in our situation the functions of interest are not only real but even complex analytic with holomorphic extensions to a neighborhood of some strip $\abs{\im{z}} \leq \delta$, for $\mathcal{C}_\delta^\omega(\mathbb{T}, \mathbb{C})$ we choose the topology induced by the norm $\vert \vert f \vert \vert_\delta := \sup_{ z \in \abs{\im{z}} \leq \delta} \abs{f(z)}$. 

More generally, if $K \subset \mathbb{C}$ compact, we equip the space of functions holomorphic on a neighborhood of $K$ with the norm $\vert \vert f \vert \vert_K := \sup_{z \in K} \abs{f(z)}$; the resulting topological space shall be denoted by $\mathfrak{A}(K)$.

We amend that in \cite{A} the authors chose a weaker topology for $\mathcal{C}_\delta^\omega(\mathbb{T}, \mathbb{C})$, induced by $\vert \vert f \vert \vert_\mathbb{T} := \sup_{ z \in \mathbb{T}} \abs{f(z)}$ resulting however in the need to fix the determinant $d(x)$ of the analytic cocycles under consideration. 

As we will show for {\em{complex}} analytic functions, (\ref{eq_loj_orig}) together with the desired uniformity of the constants will follow from basic properties of holomorphic functions.

Suggested by (\ref{eq_loj_orig}), we introduce:
\begin{definition}
Fix $\delta>0$. We say that $g \in \mathcal{C}_\delta^\omega(\mathbb{T},\mathbb{C})$ satisfies an $(\alpha,\epsilon_0)$-transversality condition if (\ref{eq_loj_orig}) holds for given exponent  $0<\alpha \leq 1$ and $\epsilon_0>0$. We denote the class of such functions in $ \mathcal{C}_\delta^\omega(\mathbb{T},\mathbb{C})$ by $\mathcal{T}_\alpha^{\epsilon_0}$.
\end{definition}
\begin{remark}
\begin{itemize}
\item[(i)] Clearly, $\mathcal{T}_\alpha^{\epsilon_0} \subseteq \mathcal{T}_\beta^{\epsilon_1}$ if $\beta < \alpha$ and $\epsilon_1 \leq \epsilon_0$.
\item[(ii)] If $f \in
  \mathcal{C}_\delta^\omega(\mathbb{T},\mathbb{C})$ has no zeros on
  $\mathbb{T}$, $f \in \mathcal{T}_\alpha^{\epsilon_0}$ for all $0 < \alpha,\leq 1$, some $\epsilon_0 (\alpha).$ 
\end{itemize}
\end{remark}

For $g \in \mathfrak{A}(K)$ not identically zero, let $\mathcal{N}(g;K)\in \mathbb{N}_0$ denote the number of zeros of $g$ on $K$ counting multiplicity. Here, $\mathbb{N}_0$ is the non-negative integers. We note the following simple fact about holomorphic functions. 
\begin{prop} \label{prop_semicont}
Let $K \subseteq \mathbb{C}$ compact with piecewise $\mathcal{C}^1$-boundary. 
\begin{itemize}
\item[(i)]  $\mathcal{N}(.;K): \mathfrak{A}(K)\setminus\{0\} \to
  \mathbb{N}_0$ is upper-semicontinuous. Moreover, if $f$ has no zeros
  on $\partial K,$ then $\mathcal{N}(g;K)$ is constant in a neighborhood of $f.$
\item[(ii)] For $j \in \mathbb{N}_0$, $g \in \mathfrak{A}(K)$ and $z \in K^\circ$ let
\begin{equation}
a_j(g,z):= \frac{1}{2\pi i} \int_{\partial K} \dfrac{g(\zeta)}{(\zeta - z)^{j+1}} \ud \zeta ~\mbox{.}\label{aj}
\end{equation}
Then, $a_j(g,z)$ is jointly continuous on $\mathfrak{A}(K) \times K^\circ$.
\end{itemize}
\end{prop}

\begin{proof}
Part (ii) follows trivially from 
\begin{equation}
\abs{a_j(g,z_0) - a_j(f, z_1)} = j! \abs{g^{(j)}(z_0) - f^{(j)}(z_1)} ~\mbox{.}
\end{equation}

To prove (i), let $f \in \mathfrak{A}(K)\setminus\{0\}$. It suffices to show $\mathcal{N}(g;K) \leq \mathcal{N}(f;K)$ for $g$ in a neighborhood of $f$. We distinguish the following two cases.
\begin{description}
\item[Case 1]
If $f$ has {\em{no}} zeros on $\partial K$, it is well known that 
\begin{equation} \label{eq_zerocount}
\mathcal{N}(f;K) =  \dfrac{1}{2 \pi i} \int_{\partial K} \dfrac{f^\prime(z)}{f(z)} \ud z ~\mbox{.}
\end{equation}
Using holomorphicity, if $g_\alpha \to f$ also $g^\prime_\alpha \to f^\prime$. Moreover, since $f$ has no zeros on $\partial K$, the same will eventually hold for $g_\alpha$. In particular,
the analogue of (\ref{eq_zerocount}) eventually expresses $\mathcal{N}(g_\alpha,K)$. Thus by bounded convergence we obtain, $\mathcal{N}(g_\alpha;K) = \mathcal{N}(f;K)$ eventually.

\item[Case 2]
If $f$ does have zeros on $\partial K$, there exists a compact neighborhood $U$ of $\partial K$ such that $f$ is holomorphic on a neighborhood of U and has no zeros on $U\setminus \partial K$.  Applying above considerations separately to $K \setminus U^\circ$ and $U$ we obtain $\mathcal{N}(g;K) \leq \mathcal{N}(f;K)$ for every $g$ sufficiently close to $f$.
\end{description}
\end{proof}

For later use we mention the following simple consequence, easily obtained by separating zeros by arbitrarily small closed balls.
\begin{coro} \label{coro_zeros}
Let $K \subset \mathbb{C}$ compact. For $f \in \mathfrak{A}(K)\setminus \{0\}$ let $\mathfrak{Z}(f;K)$ denote the set of zeros of $f$ on $K$. Then, $\mathfrak{Z}(.;K)$ is continuous in the Hausdorff metric.
\end{coro}

For $f \in \mathcal{C}_\delta^\omega(\mathbb{T},\mathbb{C})$ let $l(f)$ denote the maximal multiplicity of the distinct zeros of $f$ on $\mathbb{T}.$ As we shall argue, Proposition \ref{prop_semicont} implies openness of $\alpha$-transversality:
\begin{theorem} \label{thm_opentrans}
For fixed $\delta >0$,
\begin{itemize}
\item[(i)] Suppose $f \in \mathcal{C}_\delta^\omega(\mathbb{T},\mathbb{C})$ does not vanish identically but possesses zeros on $\mathbb{T}$. 
Then, $f$ satisfies an $(\alpha,\epsilon_0)$-transversality condition with $\alpha = \left(l(f)^{-1}\right)^{-}$. 
\item[(ii)] Let $f \in \mathcal{C}_\delta^\omega(\mathbb{T},\mathbb{C})$ then for any $0\leq \alpha < l(f)^{-1}$, there is $\epsilon_0(f,\alpha)$ such that
$f \in\mathrm{Int} \mathcal{T}_\alpha^{\epsilon_0}$ wrt  $\mathcal{C}_\delta^\omega(\mathbb{T},\mathbb{C})$.
\end{itemize}
\end{theorem}

\begin{proof}
Let $f \in \mathcal{C}_\delta^\omega(\mathbb{T},\mathbb{C})\setminus\{0\}$ be fixed. Clearly, if $l(f)=0$ so is $l(g)$ for any $g \in \mathcal{C}_\delta^\omega(\mathbb{T},\mathbb{C})$ sufficiently close to $f$ in which case the theorem becomes trivial. Hence, without loss, we may assume $l(f) \geq 1$. 

Let $x_1, \dots, x_n$ be the distinct zeros of $f$ with multiplicities respectively, $l_1, \dots, l_n$. Choose $K$, a compact neighborhood of $\mathbb{T}$, such that $f(z) \neq 0$ on $K\setminus \mathbb{T}$. Separating the zeros by closed balls of some appropriate radius $r$, there exists $\eta >0$ such that $\vert \vert g -  f \vert \vert_\delta < \eta$ guarantees $\abs{g(z)} >  \frac{1}{2} \min_{K \setminus \cup_{j=1}^{n} B(x_j,r)} \abs{f(z)} =: m > 0$ on $K \setminus \cup_{j=1}^{n} B(x_j,r)$. In particular, possible zeros on $K$ of any such function $g$ will lie in $\cup_{j=1}^{n} B(x_j,r)$. Using Proposition \ref{prop_semicont}(ii), $\eta$ can be chosen small enough to also ensure $\mathcal{N}(g;K) = \mathcal{N}(f;K)$.

We claim
\begin{lemma} \label{lem_locLoj}
There exist $0< \eta^\prime <\eta, 0 < \eta^{''} < r$ and $\kappa>0$ such that uniformly over $\vert \vert g - f \vert \vert_\delta < \eta^\prime$,
\begin{equation} \label{eq_locLoj}
\abs{g(z)} \geq \kappa \abs{p_{g,\tilde{z}}(z)} ~\mbox{,} ~\abs{z - x_j}<\eta^{''} ~\mbox{,}
\end{equation}
for  every zero $\tilde{z}$ of $g$ on $K$, $\tilde{z} \in B(x_j,r)$, and some monic polynomial $p_{g, \tilde{z}}(z)$ of degree at most $l_j,$ with $p_{g,\tilde{z}}(\tilde{z})=0$. Moreover,
as $\Vert g - f \Vert_\delta \to 0$, $p_{g,\tilde{z}}(z) \to (z-x_j)^{l_j}$ uniformly on compact subsets of $\mathbb{C}$.
\end{lemma}
\begin{proof}
For $g$ with $\vert \vert g - f \vert \vert < \eta$ using holomorphicity we can write
\begin{equation} \label{eq_locLoj1}
g(z) = \left( 1 + h_{g,\tilde{z}}(z) \right) \sum_{k=l(g,\tilde{z})}^{l_j} a_k(g,\tilde{z}) (z - \tilde{z})^k ~\mbox{,}
\end{equation}
locally about a zero $\tilde{z}$ of $g$, $\tilde{z} \in B(x_j,r)$, and
some holomorphic $h_{g,\tilde{z}}$ with $h_{g,\tilde{z}}=o(1)$
uniformly as $z \to \tilde{z},$ where $a_k$ are as in (\ref{aj}). 

In fact, using Cauchy estimates and Proposition
\ref{prop_semicont}(i), there is $0 < \eta^\prime < \eta$ and $0 <
\eta^{''} < r$ such that uniformly over $\vert \vert g - f \vert
\vert_\delta < \eta^\prime$ and for every zero $\tilde{z}$ of $g$ on
$K$ with $\tilde{z} \in B(x_j, r)$, $\abs{h_{g,\tilde{z}}(z)} < 1/2$
if $\abs{z - x_j} < \eta^{''}$. By Proposition
\ref{prop_semicont}(ii) (note that all zeros of $g$ are in $K^{\circ}$),
the polynomial in (\ref{eq_locLoj1}) is uniformly close on compact subsets of $\mathbb{C}$ to $a_{l_j}(f, x_j) (z - x_j)^{l_j}$. In particular, this implies that 
$\eta^\prime$ can be chosen small enough so that $\vert a_{l_j}(g,\tilde{z}) \vert > 1/2 \min_{1 \leq j \leq n} \abs{a_{l_j}(f, x_j)} > 0,$ which yields the claim.
\end{proof}

To complete the proof of Theorem \ref{thm_opentrans} we use the following well-known theorem due to P\'olya \cite{BB}.
\begin{theorem}[P\'olya] \label{thm_ploya}
Let $p_n(z)$ be a complex monic polynomial of degree at most $n \geq 1$. Then, for $\epsilon > 0$
\begin{equation}
\mu_L \left( \{ x \in \mathbb{R}: \abs{p_n(x+iy)} \leq \epsilon \} \right) \leq 4 \epsilon^{1/n} ~\mbox{.}
\end{equation}
Here, $\mu_L$ denotes the Lebesgue mesure.
\begin{remark}
The proof in \cite{BB} actually shows that under the hypotheses of Theorem \ref{thm_ploya} one has
\begin{equation}
\mu_L \left( \{ x \in \mathbb{R}: \abs{p_n(x+iy)} \leq \epsilon \} \right) \leq 2^{2-1/n} \epsilon^{1/n} ~\mbox{.}
\end{equation}
\end{remark}
\end{theorem}

Let $\epsilon_0 : = \frac{1}{2} \min_{z \in K \setminus \cup_{j=1}^{n} B(x_j, \eta^{''})} \abs{f(z)} > 0$. If $\Vert g - f \Vert_\delta < \eta^\prime$, using Lemma \ref{lem_locLoj} we conclude for $0< \epsilon < \epsilon_0$:
\begin{equation}
\begin{split}
\left\{ x \in \mathbb{T}: \right. & \left. \abs{g(x)} < \epsilon \right\} \subseteq \bigcup_{\tilde{z} \in \mathfrak{Z}(g;K)} \left\{ x \in \mathbb{T}: \abs{p_{g,\tilde{z}}(x)} < \epsilon/\kappa \right\} ~\mbox{.}
\end{split}
\end{equation}

Applying Theorem \ref{thm_ploya} we thus obtain 
\begin{equation} \label{eq_thmunifloj_1}
\mu_L\left( \left\{ x \in \mathbb{T}: \abs{g(x)} \leq \epsilon \right\} \right) \leq 4 \mathcal{N}(f;K) \left(\frac{\epsilon}{\kappa}\right)^{1/l(f)} ~\mbox{,}
\end{equation}
for $0< \epsilon < \epsilon_0$ for all $g$ with $\Vert g - f \Vert_\delta < \eta^\prime$.

In particular, for  $0<\gamma<1/l(f)$ and $0<\alpha_\gamma=1/l(f) - \gamma$, we conclude that $g \in \mathcal{T}_{\alpha_\gamma}^{\epsilon_\gamma}$ with $\epsilon_\gamma = \min\left\{ \epsilon_0, \left( \frac{4 \mathcal{N}(f;K)}{\kappa^{1/l(f)}}\right)^{1/\gamma} \right\}$. 
\end{proof}
\begin{remark}\begin{enumerate}
\item We mention that P\'olya's Theorem \ref{thm_ploya} was used to deal with a possible ``collapse of zeros'' as $g \to f$. If we restrict to functions $g$ having the same number of {\em{distinct}} zeros as $f$ in {\em{some}} neighborhood of $\mathbb{T}$, then (\ref{eq_thmunifloj_1}) can be obtained directly from Lemma \ref{lem_locLoj} since in this case (\ref{eq_locLoj}) simplifies to
\begin{equation}
\abs{g(z)} \geq \kappa \abs{z}^{l(f)} ~\mbox{,} ~\abs{z - x_j}<\eta^{''} ~\mbox{.}
\end{equation}
\item An alternative proof can be obtained using the Cartan's estimate.
\end{enumerate}
\end{remark}

We conclude this section with the following Lemma closely related to Proposition \ref{prop_semicont}, which will come handy in the proof of the uniform large deviation bound:
\begin{lemma} \label{eq_lemint}
For $f \in \mathcal{C}_\delta^\omega(\mathbb{T},\mathbb{C}),$ not vanishing identically let
\begin{equation}
I(f) := \frac{1}{2 \pi} \int \log \abs{f(x)}\ud \mu(x) ~\mbox{.} 
\end{equation}
Then $I$ is continuous on $\mathcal{C}_\delta^\omega(\mathbb{T},\mathbb{C}) \setminus \{0\}$ w.r.t. $\vert \vert . \vert \vert_\delta$.
\end{lemma}
\begin{proof}
Fixing $f \in \mathcal{C}_\delta^\omega(\mathbb{T},\mathbb{C})$, let $x_1, \dots x_n$ denote the zeros of $f$ on $\mathbb{T}$ counting multiplicities. Fix a compact neighborhood $U$ of $\mathbb{T}$ such that $f$ extends holomorphically to a neighborhood of $U$ and $f(z) \neq 0$ for $z \in U\setminus \mathbb{T}$. Letting
\begin{equation} 
g(z)=\dfrac{f(z)}{\prod_{j=1}^{n} (z-\mathrm{e}^{2\pi i x_j})} ~\mbox{,}
\end{equation}
we obtain $I(g) = I(f)$ (which is a proof of $\log \abs{f} \in L^1(\mathbb{T})$). Here, we made use of the identity,
\begin{equation} \label{eq_ident}
\frac{1}{2 \pi} \int_{0}^{1} \log \abs{1 - \mathrm{e}^{2 \pi i x_j}} \ud x = 0 ~\mbox{.}
\end{equation}

If $f_\alpha \to f$ then eventually $f_\alpha$ will have no zeros on $\partial U$; hence
letting $w_1, \dots, w_m$, denote the zeros of $f_\alpha$ on $U^\circ$ counting mulitplicities, in analogy to $g$ we can define $g_\alpha$, dividing out these zeros.

By Proposition \ref{prop_semicont}(i), $m = n$ eventually as $f_\alpha
\to f$; hence, also making use of Corollary \ref{coro_zeros} and
Proposition \ref{prop_semicont}(ii) when treating small neighborhoods
of $x_j$, we deduce $g_\alpha \to g$ uniformly on $U$. Finally, note that by Jensen's formula and (\ref{eq_ident}), $I(g_\alpha) = I(f_\alpha)$.
\end{proof}

\section{Uniform large deviation bound} \label{sec_unifdevbd}

We first fix some notation. Given an analytic cocycle $(\beta, D(x))$ we define its iterates 
\begin{equation}
D_n(x):=D(x+(n-1)\beta) \dots D(x) ~\mbox{,} ~n \in \mathbb{N} ~\mbox{.}
\end{equation}

Moreover, for $n \in \mathbb{N}$ let
\begin{equation} \label{ln}
L_n(\beta,D):=\dfrac{1}{n} \int \log \norm{D_n(x)} \ud \mu(x) ~\mbox{,}
\end{equation}
denote the $n$th approximate of $L(\beta, D)$. 

We then claim the following uniform version of the crucial Lemma 1 in  \cite{A}:
\begin{theorem}[Uniform large deviation bound for analytic cocycles (ULDB)] \label{thm_LDB}
Fix $\beta$ Diophantine and $D(x) \in \mathcal{C}_\delta^\omega(\mathbb{T}, M_2(\mathbb{C}))$ with $d(x)$ not vanishing identically. Let $p/q$ denote an approximant of $\beta$: $\abs{\beta - \frac{p}{q}} < \frac{1}{q^2}$ with $(p,q) = 1$. 

There exist $\gamma(D)>0$ and constants $0< c, C < \infty$ such that for $0<\kappa<1, ~n > (C \kappa^{-2} q)^\eta$ with $\eta = \eta(\beta) > 1$ and for $q$ sufficiently large, {\em{uniformly over}} $\vert \vert \tilde{D} - D \vert \vert_\delta < \gamma$,
\begin{equation}
\mu \left \{\left \vert \dfrac{1}{n} \log \norm{\tilde{D}_n(x)} - L_n(\beta,\tilde{D}) \right \vert > \kappa \right \} < \mathrm{e}^{-c \kappa q} ~\mbox{.}
\end{equation}
\end{theorem}

\begin{proof}
For $\vert \vert \tilde{D}-D \vert \vert_\delta < \gamma$ we set $\tilde{u}_n = \tilde{u}_n(\beta,\tilde{D};x):=\frac{1}{n}\log \norm{\tilde{D}_n(x)}$, $n \in \mathbb{N}$. By hypotheses, $\tilde{u}_n$ extends to a subharmonic function on the $\delta$-strip about $\mathbb{T}$. Notice that due to possible zeros of $\tilde{d}(x)$, in general, $\tilde{u}_n$ will not be bounded.

We shall deal with the unboundedness of $\tilde{u}_n$, introducing appropriate cut-offs. To this end choose $0<A<\infty$ such that
$\inf_{\vert \vert \tilde{D} - D \vert \vert_\delta < \gamma}
I(\tilde{d}) > -2 A$ (which is finite by a compactness
argument). Here, $I(.)$ is defined as in Lemma \ref{eq_lemint}. For $n
\in \mathbb{N}$, let $\tilde{w}_n(z):=\max\{\tilde{u}_n(z), -A\}$; thereby we obtain a family $\{\tilde{w}_n, n \in \mathbb{N}\}$ of subharmonic functions on the $\delta$-strip of $\mathbb{T}$, uniformly bounded in $n$ and over $\vert \vert \tilde{D} -D\vert \vert_\delta < \gamma$.

The strategy to prove the ULDB is to estimate deviations of the individual terms in
\begin{eqnarray} \label{eq_strategy}
\vert \dfrac{1}{n} \log \norm{\tilde{D}_n(x)} - L_n(\beta,\tilde{D})  \vert <   \vert \tilde{u}_n(x) - \tilde{w}_n(x)  \vert +   \nonumber \\ 
  \vert \tilde{w}_n(x) - \langle \tilde{w}_n \rangle  \vert +  \vert \langle \tilde{w}_n \rangle - L_n(\beta,\tilde{D})  \vert ~\mbox{.}
\end{eqnarray}
Here, and following we use the notation $\langle f \rangle$ to denote the 0th Fourier coefficient of a function $f \in L^1(\mathbb{T})$.

We start by estimating the 2nd contribution in (\ref{eq_strategy}). For any $R>0$ we can write
\begin{eqnarray} \label{eq_ldb1}
\vert \tilde{w}_n(x) - \langle \tilde{w}_n \rangle  \vert \leq \left \vert \tilde{w}_n(x) - \sum_{\abs{j} < R} \dfrac{R- \abs{j}}{R^2} \tilde{w}_n(x+ j \beta)           \right \vert  + \nonumber \\
\left \vert \sum_{\abs{j} < R} \dfrac{R- \abs{j}}{R^2} \tilde{w}_n(x+ j \beta) - \langle \tilde{w}_n \rangle   \right \vert =: \abs{(\mathrm{I})} + \abs{(\mathrm{II})} ~\mbox{.}
\end{eqnarray}
$R$ will be suitably chosen later.

Contribution (II) is readily controlled using the following result  \cite{C},
\begin{lemma}[Large deviation bound for bounded subharmonic functions
  \cite{C}; see also\cite{A}, p. 1888] \label{lem_ldb}
Let $v(x)$ be a bounded 1-periodic subharmonic function defined on a neighborhood of $\mathbb{R}$. Let $\abs{\beta - \frac{p}{q}} < \frac{1}{q^2}$, $(p,q)=1$ and $0 < \kappa < 1$. Then for appropriate $0< C_1, c_1 < \infty$ and for $R> C_1 \kappa^{-1} q$ we have,
\begin{equation}
\mu \left \{\left \vert \sum_{\abs{j} < R} \dfrac{R- \abs{j}}{R^2} v(x+ j \beta) - \langle v \rangle   \right \vert > \kappa\right\} < \mathrm{e}^{-c_1 \kappa q} ~\mbox{.}
\end{equation}
\end{lemma}
\begin{remark}
For a uniformly bounded family of subharmonic functions, the constants
$c_1, C_1$ can be chosen uniformly over this family \cite{A}. This in particular, applies to the family $\{\tilde{w}_n, n \in \mathbb{N} ~\mbox{and} ~\vert \vert \tilde{D} - D \vert \vert_\delta < \gamma\}$.
\end{remark}

To estimate contribution (I) in (\ref{eq_ldb1}), we establish the following
\begin{prop} \label{prop_udb1}
Uniformly over $\vert \vert \tilde{D} - D\vert \vert_\delta < \gamma$, there exists $0<c_2$ such that for any $0<\epsilon<1$ there is $0<C_2<\infty$ with
\begin{equation}
\mu \left\{\abs{ \tilde{w}_n(x) - \tilde{w}_n(x+\beta)} > \dfrac{C_2}{n^{1-\epsilon}} \right\} < \mathrm{e}^{-c_2 n^\epsilon} ~\mbox{,}
\end{equation}
for sufficiently large $n$ (only depending on $d(x)$ and $\epsilon$).
\end{prop}

\begin{proof}
Let $B:=\gamma + \vert \vert D \vert \vert_\delta $ whence $\sup_{\vert \vert \tilde{D} - D \vert \vert_\delta < \gamma} \vert \vert \tilde{D} \vert \vert_\delta < B$. By definition of $\tilde{w}_n$ we have
\begin{equation}
\abs{\tilde{w}_n(x) - \tilde{w}_n(x+\beta)} < \frac{1}{n} \left \vert \log \dfrac{ \norm{\tilde{D}_n(x)}}{\norm{\tilde{D}_n(x+\beta)}} \right \vert ~\mbox{.}
\end{equation}
For any $M \in GL_2(\mathbb{C})$, 
\begin{equation}
\norm{M^{-1}} = \dfrac{\norm{M}}{\abs{\det M}} ~\mbox{,}
\end{equation}
whence
\begin{equation} \label{eq_ldb2}
\max \left\{ \dfrac{\norm{\tilde{D}_n(x)}}{\norm{\tilde{D}_n(x+\beta)}}, \dfrac{\norm{\tilde{D}_n(x+\beta)}}{\norm{\tilde{D}_n(x)}} \right\} < B^2 \max\left\{\frac{1}{\abs{\tilde{d}(x + n\beta)}}, \frac{1}{\abs{\tilde{d}(x)}}\right\} ~\mbox{.}
\end{equation}

Let $0<\epsilon<1$. 
If $\abs{\tilde{d}(x + j \beta)} \geq \mathrm{e}^{-n^\epsilon}$ for {\em{both}} $j=0,n$ using (\ref{eq_ldb2}) we obtain
\begin{equation}
\abs{\tilde{w}_n(x) - \tilde{w}_n(x+\beta)} < \dfrac{2 \log{B}}{n} + n^{\epsilon-1} < \dfrac{C_2}{n^{1-\epsilon}} ~\mbox{.}
\end{equation}

Hence, using Theorem \ref{thm_opentrans} we estimate
\begin{eqnarray}
\mu \left\{\abs{ \tilde{w}_n(x) - \tilde{w}_n(x+\beta)} > \dfrac{C_2}{n^{1-\epsilon}} \right\} \leq \nonumber \\
\mu \left\{\abs{\tilde{d}(x+j \beta)} < \mathrm{e}^{-n^\epsilon} ~\mbox{, some} ~j \in \{0,n\} \right\} \leq 2 \mathrm{e}^{- \alpha n^\epsilon} ~\mbox{,}
\end{eqnarray}
for $n$ sufficiently large uniformly over $\vert \vert \tilde{D} - D \vert \vert_\delta < \gamma$. Here and in the following, $\alpha$ is the exponent for $d(x)$ determined by Theorem \ref{thm_opentrans} by the maximal multiplicity of the zeros of $d(x)$ on $\mathbb{T}$. Finally choosing $c_2 < \alpha$ we obtain the claim of the Proposition.
\end{proof}

We are now ready to estimate $\mu \{x: \abs{\tilde{w}_n(x) - \langle \tilde{w} \rangle} > \kappa\}$: Let $\mathfrak{X}:=\{x: \abs{\tilde{w}_n(x) - \tilde{w}_n(x+\beta)} > \frac{C_2}{n^{1-\epsilon}}\}$, where $C_2, \epsilon$ are as in Proposition \ref{prop_udb1}. Denote by $T$ the rotation by $\beta$ on $\mathbb{T}$.

If $x \in \mathbb{T}$ is such that $\cup_{j=-R+1}^{R+1} T^j x \subseteq \mathbb{T}\setminus \mathfrak{X}$, then referring to 
(\ref{eq_ldb1}) we obtain
\begin{equation} 
\abs{(I)} = \abs{\tilde{w}_n - \sum_{\abs{j} < R} \dfrac{R-\abs{j}}{R^2} \tilde{w}_n(x+j \beta)} < \dfrac{C_2}{n^{1-\epsilon}}R ~\mbox{.}
\end{equation}
In particular, choosing $R < \frac{\kappa n^{1-\epsilon}}{2 C_2}$ implies that for such $x$ we have $\abs{(\mathrm{I})} < \kappa/2$.

The largeness condition on $R$ from Lemma \ref{lem_ldb} will also be taken care of when letting $C_1 \kappa^{-1} q < R < \frac{\kappa n^{1-\epsilon}}{2 C_2}$; this is accommodated choosing  $n > N$ with 
\begin{equation}
N:= \left( 2 C_1 C_2 \kappa^{-2} q\right)^{\frac{1}{1-\epsilon}} ~\mbox{.}
\end{equation}

Thus fixing
\begin{equation}
R := \frac{1}{2} \left( \frac{\kappa N^{1-\epsilon}}{2 C_2} + C_1 \kappa^{-1} q \right) ~\mbox{,}
\end{equation}
and using Lemma \ref{lem_ldb} and Proposition \ref{prop_udb1}, we have for $n > N$
\begin{eqnarray} \label{eq_ldb14}
\mu \{x: \abs{\tilde{w}_n(x) - \langle \tilde{w}_n \rangle} > \kappa\} \leq \mu_L\{x: \abs{(\mathrm{I})} > \kappa/2, \cup_{j=-R+1}^{R} T^{j} x \subseteq \mathbb{T}\setminus \mathfrak{X} \} + \nonumber \\
 \mu \{x: \abs{(\mathrm{I})} > \kappa/2, T^{j} x \in \mathfrak{X} ~\mbox{some} ~-R+1 \leq j \leq R \} + \mu_L\{x:\abs{(\mathrm{II})} > \kappa/2 \} \leq \nonumber \\
\mathrm{e}^{-c_1 \frac{\kappa}{2} q} + 2 R \mathrm{e}^{-c_2 \frac{\kappa}{2} q} \leq \mathrm{e}^{-c_3 \kappa q} ~\mbox{,}
\end{eqnarray}
for suitable $c_3 > 0 $. This completes the estimate of the 2nd contribution in (\ref{eq_strategy}).

Consider now the third term in in (\ref{eq_strategy}). Set
\begin{eqnarray} \label{eq_ldb9}
\tilde{\mathfrak{Y}}_n:=\{x: \tilde{w}_n(x) \neq \tilde{u}_n(x)\} = \{x: \norm{\tilde{D}_n(x)} < \mathrm{e}^{-n A}\} \\
\subseteq  \{x: (\prod_{j=0}^{n-1} \abs{\tilde{d}(x+j \beta)})^\frac{1}{2} < \mathrm{e}^{-n A}\} ~\mbox{.} \label{eq_ldb3}
\end{eqnarray}

In order to analyze the product of analytic functions occurring in (\ref{eq_ldb3}) we establish the following:
\begin{prop} \label{pro_ldb2}
Let $f \in \mathcal{C}_\delta^\omega$ not vanishing identically and let $\beta$ be fixed satisfying the Diophantine condition (\ref{eq_diophantine}). Then, there exist $0 < C_3 = C_3(\beta)$ and $0 < C_4 = C_4(f)$ such that for $n \in \mathbb{N}$ we have
\begin{equation}
\abs{\frac{1}{n} \sum_{j=1}^{n} \log \abs{f(x + j \beta)} - \langle \log \abs{f} \rangle} \leq C_3 \mathcal{N}(f;\mathbb{T}) n^{-1/r(\beta)} \log^2(n) \left\vert\min_{1\leq j \leq n} \log \abs{f(x + j \beta)}\right\vert + \dfrac{C_4}{n}   ~\mbox{.}
\end{equation}
The constant $C_4(\tilde f)$ can be chosen uniformly over $\vert \vert \tilde{f} - f\vert \vert_\delta < \epsilon$ for $\epsilon>0$ sufficiently small.
\end{prop}
\begin{remark} \label{rem_pro_ldb2}
\begin{itemize}
\item[(i)] It is through this Proposition that a Diophantine condition is imposed on $\beta$  in Theorem \ref{thm_LDB}.
\item[(ii)] Using upper-semicontinuity of $\mathcal{N}(.;\mathbb{T})$ established in Proposition \ref{prop_semicont}, \newline $\mathcal{N}(\tilde f;\mathbb{T})$ can be chosen uniformly over $\vert \vert \tilde{f} - f\vert \vert_\delta < \epsilon$ for sufficiently small $\epsilon>0$.
\end{itemize}
\end{remark}
We mention that Proposition \ref{pro_ldb2} improves on and provides a
uniform version of Proposition C from \cite{A}, originally proven in \cite{Q}. The statement given here was inspired by estimates on trigonometric products which played an important role in \cite{R} (see Sec. 9.2 therein). 

\begin{proof}
First, write $f$  as 
\begin{equation} \label{eq_ldb7}
f(z) = g(z) \prod_{j=1}^{\mathcal{N}(f;\mathbb{T})} (\mathrm{e}^{2 \pi i x_j} - z)  
\end{equation}
on a compact neighborhood $U$ where $f$ is holomorphic and exhibits its {\em{only}} zeros $\{x_j, 1 \leq j \leq \mathcal{N}(f;\mathbb{T})\}$ on $\mathbb{T}$ (counting multiplicities). In particular $g$ is holomorphic on a neighborhood of $U$ and $\mathcal{N}(g;U) = 0$.

We first establish the zero free version of Proposition \ref{pro_ldb2}.
\begin{lemma}[Zero-free version of Proposition \ref{pro_ldb2}] \label{lem_pro_ldb2}
Let $\beta \in [0,1)$ be a fixed Diophantine number satisfying condition (\ref{eq_diophantine}) and $g$ a zero free function on a compact neighborhood $U$ of $\mathbb{T}$. Then for $n \in \mathbb{N}$,
\begin{equation} \label{eq_ldb8}
\vert \frac{1}{n} \sum_{j=1}^{n} \log \abs{g(x+j \beta)} - \langle \log \abs{g} \rangle \vert \leq \frac{C_4}{n} ~\mbox{.}
\end{equation}
Here, $C_4 = C_4(g)$ can be chosen uniformly over $\vert \vert \tilde{g} - g \vert \vert_U < \epsilon$ for $\epsilon>0$ sufficiently small (only depending on $\min_{z \in U} \abs{g(z)}$).
\end{lemma}
\begin{remark}
\begin{itemize}
\item[(i)] The main purpose here is to convince the reader that $C_4$ can be chosen {\em{uniformly}}; the mere rate of convergence for the zero free situation is a standard fact from harmonic analysis.
\item[(ii)] Uniformity of $C_4(\tilde f)$ follows from Lemma \ref{lem_pro_ldb2} since $f_\alpha \to f$ uniformly on $U$ implies uniform convergence of the respective zero-free functions (for a simple argument see the proof of Lemma  \ref{eq_lemint}).
\end{itemize}
\end{remark}

\begin{proof}
Find $\epsilon_0>0$ such that $\abs{\tilde{g}(z)} > 1/2 \min_{z \in U} {\abs{g(z)}}>0$ for $\vert \vert \tilde{g} -g \vert \vert_U < \epsilon_0$. In particular, $\vert \vert \tilde{g}-g \vert \vert_U < \epsilon_0$ implies $\mathcal{N}(\tilde{g};U) = \mathcal{N}(g;U) = 0$ whence letting $\tilde{G}:=\log \abs{\tilde{g}}$ we obtain that $\tilde{G}$ is harmonic on a neighborhood of $U$. Hence, $\tilde{G} = \sum_{k \in \mathbb{Z}} \tilde{G}_k \mathrm{e}^{2 \pi i k x}$ converges absolutely and uniformly on $\mathbb{T}$ and for $n \in \mathbb{N}$
\begin{equation}
\frac{1}{n} \sum_{j=0}^{n-1} \tilde{G}(x + j \beta) - \tilde{G}_0 = \frac{1}{n} \sum_{k \in \mathbb{Z}\setminus \{0\}} \tilde{G}_k \mathrm{e}^{2 \pi i k x} \dfrac{1 - \mathrm{e}^{2 \pi i k n \beta}}{1- \mathrm{e}^{2 \pi i k \beta}} ~\mbox{.}
\end{equation}

Making use of  Eq. (\ref{eq_diophantine}), results in 
\begin{equation} \label{eq_estim_ergodic}
\left \vert \frac{1}{n} \sum_{j=0}^{n-1} \tilde{G}(x + j \beta) - \tilde{G}_0 \right \vert \leq \dfrac{2 b(\beta)}{n} \sum_{k \in \mathbb{Z}} \abs{\tilde{G}_k} k^{M} ~\mbox{,}
\end{equation}
where $M=M(\beta):= \lceil r(\beta) \rceil$. 

As in Katznelson \cite{Y}, let $A(\mathbb{T}) \subset L^1(\mathbb{T})$ denote the class of 1-periodic functions with absolutely converging Fourier series equipped with
the norm $\lVert f \rVert_{A(\mathbb{T})}:= \sum_{k \in \mathbb{Z}} \abs{f_k}$. $A(\mathbb{T})$ is a homogeneous Banach space of $L^1(\mathbb{T})$ isomorphic to $\mathit{l}^1(\mathbb{Z})$.

We employ the following standard fact :
\begin{prop}\cite{Y} \label{prop_katz}
Let $f \in L^1(\mathbb{T})$ be absolutely continuous with $f^\prime \in L^2(\mathbb{T})$. Then $f \in A(\mathbb{T})$ and 
\begin{equation}
\lVert f \rVert_{A(\mathbb{T})} \leq \lVert f \rVert_{L^1(\mathbb{T})} + \left( 2 \sum_{n=1}^{\infty} \frac{1}{n^2} \right)^{1/2} \lVert f^\prime \rVert_{L^2(\mathbb{T})} ~\mbox{.}
\end{equation}
\end{prop}

In summary, Proposition \ref{prop_katz} and (\ref{eq_estim_ergodic}) imply
\begin{equation}
\left \vert \frac{1}{n} \sum_{j=0}^{n-1} \tilde{G}(x + j \beta) - \tilde{G}_0 \right \vert \leq \dfrac{4 b(\beta)}{n}  \left( \lVert \tilde{G}^{(M)}  \rVert_\mathbb{T} + \lVert \tilde{G}^{(M+1)}  \rVert_{\mathbb{T}}  \right) ~\mbox{.}
\end{equation}
Finally, it is a basic property of harmonic functions that $\tilde{G} \to G$ uniformly on $U$ implies $\tilde{G}^{(m)} \to G^{(m)}$ for any $m \in \mathbb{N}$ \cite{Z}, which yields the claim.
\end{proof}

Thus we are left with analyzing the rate of convergence of the terms $\log\abs{\mathrm{e}^{2 \pi i x_j} - z}$ in a Caesaro mean.
We employ the following Lemma proven in \cite{R}
\begin{lemma} \label{lem_ldb3}
Let $\beta$ be irrational. For $n \in \mathbb{N}$, let $p_n/q_n$ denote the $n$th approximant of $\beta$. Then, if $1 \leq k_0 \leq q_n$ is determined by
\begin{equation}
\left \vert \sin \left( \dfrac{2 \pi (x-x_0) + k_0 \beta }{2} \right) \right \vert := \min_{1 \leq k \leq q_n} \left \vert \sin \left( \dfrac{2 \pi (x-x_0) + k \beta }{2} \right) \right \vert ~\mbox{,}
\end{equation}
we have:
\begin{eqnarray}
\left \vert  \sum_{\stackrel{k=1}{k \neq k_0}}^{q_n} \log \left \vert  \sin \left( \dfrac{2 \pi (x-x_0) + k \beta }{2} \right) \right \vert +  (q_n -1) \log(2)   \right \vert  = \\
\left \vert  \sum_{\stackrel{k=1}{k \neq k_0}}^{q_n} \log \left \vert  \mathrm{e}^{2 \pi i (x + k \beta)} - \mathrm{e}^{2 \pi i x_0}  \right \vert   \right \vert             
 < C_5 \log(q_n) ~\mbox{.}
\end{eqnarray}
Here, $C_5 = C_5(\beta)$.
\end{lemma}

Fix $1 \leq j \leq \mathcal{N}(f;\mathbb{T})$. For $n \in \mathbb{N}$
arbitrary let $s=s(n)\geq0$ such that $q_s \leq n < q_{s+1}$. By
successive division represent $n$ as $n=\sum_{k=0}^{s} l_k
q_k$. Recall that by (\ref{eq_diophantine}) \cite{S},
\begin{equation}
q_{k+1} \leq \frac{2 \pi}{b(\beta)} q_k^{r(\beta)} ~\mbox{,} ~k \in \mathbb{N}_0 ~\mbox{.}
\end{equation}
In particular, this allows to control the divisors 
\begin{equation} \label{eq_ldb5}
l_k \leq \frac{q_{k+1}}{q_k} \leq \left(\frac{2 \pi}{b}\right)^{1/r} q_{k+1} ~\mbox{,} ~k<s ~\mbox{, and}  ~l_s \leq \frac{n}{q_s} \leq \left(\frac{2 \pi}{b}\right)^{1/r} n^{1-1/r} ~\mbox{.}
\end{equation}
Let $1 \leq k_0 \leq n$ such that 
\begin{equation}
\left \vert \sin \left( \dfrac{2 \pi (x-x_j) + k_0 \beta }{2} \right) \right \vert := \min_{1 \leq k \leq q_n} \left \vert \sin \left( \dfrac{2 \pi (x-x_j) + k \beta }{2} \right) \right \vert ~\mbox{.}
\end{equation}
Making use of Lemma \ref{lem_ldb3} and (\ref{eq_ldb5}) yields
\begin{eqnarray}
\left \vert \sum_{\stackrel{k=1}{k \neq k_0}}^{n} \log \left \vert   \mathrm{e}^{2 \pi i (x+k\beta)} - \mathrm{e}^{2 \pi i x_j}  \right \vert   \right \vert \leq C_5 \sum_{k=0}^{s} \log{q_k} \frac{q_{k+1}}{q_k} \leq \nonumber \\
C_5 \log q_s \left( \left(\frac{2 \pi}{b}\right)^{1/r} q_s^{1-1/r} \frac{2}{\log{2}} \log q_s  + \log q_s  \left(\frac{2 \pi}{b}\right)^{1/r} n^{1-1/r} \right) \leq\nonumber \\
C_6(\beta) \log^2(n) n^{1-1/r} ~\mbox{.} \label{eq_ldb6}
\end{eqnarray}
Here, we also used a general fact that allows to control $s$ by $s(n) \leq \frac{2 \log q_s}{\log 2}$ (valid for any irrational $\beta$) \cite{S}.

Finally, combining (\ref{eq_ldb7}), (\ref{eq_ldb8}), and (\ref{eq_ldb6}) implies the claim of Proposition \ref{pro_ldb2}.
\end{proof}

We can now estimate $\mu\{ \tilde{\mathfrak{Y}}_n \}$ (see (\ref{eq_ldb9})). To this end, suppose $x \in \tilde{\mathfrak{Y}}_n,$ then employing Proposition \ref{pro_ldb2} results in
\begin{eqnarray}
\min_{0 \leq j \leq n-1} \abs{\tilde{d}(x+j \beta)} <
\mathrm{e}^{-n^\epsilon} 
\end{eqnarray}
for $n$ sufficiently large
and  any $0 < \epsilon < r(\beta)^{-1} \leq 1$. Hence, by theorem \ref{thm_opentrans}
\begin{equation} \label{eq_ldb11}
\mu \{\tilde{\mathfrak{Y}}_n\} \leq \mu \{x: \min_{0 \leq j \leq n-1} \abs{\tilde{d}(x+j \beta)} < \mathrm{e}^{-n^\epsilon} \} \leq n \mathrm{e}^{-n^\epsilon \alpha} \leq \mathrm{e}^{-n^\epsilon c_2} ~\mbox{,}
\end{equation}
for $\alpha$ determined by $d,$ $\gamma=\gamma(d)$ sufficiently small
and $n$ sufficiently large, uniformly over $\vert \vert \tilde{D} - D \vert \vert_\delta < \gamma.$

Since $\norm{M}^2 \geq \abs{\det{M}}$ for $M \in M_2(\mathbb{C})$, 
\begin{equation} \label{eq_ldb10}
\vert \langle \tilde{w}_n \rangle - L_n(\beta, \tilde{D}) \vert \leq \frac{1}{n} \int_{\tilde{\mathfrak{Y}}_n} \log  \left \vert \dfrac{\mathrm{e}^{-n A}}{\norm{\tilde{D}_n(x)}} \right \vert \ud x \leq  \frac{1}{n} \int_{\tilde{\mathfrak{Y}}_n} \log  \left( \dfrac{\mathrm{e}^{-n A}}{\prod_{j=0}^{n-1} \abs{\tilde{d}(x+j \beta)}^\frac{1}{2}    } \right) \ud x ~\mbox{.}
\end{equation}

\begin{lemma} \label{lem_ldb4}
Uniformly over $\vert \vert \tilde{D}-D \vert \vert_\delta <\gamma$
there exists $\alpha=\alpha(d)$ and $0<C_7=C_7(\alpha)<\infty$ such that for $i,j \in \{0, \dots, n-1\}$,
\begin{equation}
\left \vert \int_{\abs{\tilde{d}_i(x)} < \epsilon}  \log \abs{\tilde{d}_j(x)} \ud x \right \vert \leq C_7 \epsilon^\alpha \abs{\log \epsilon} ~\mbox{,}
\end{equation}
for sufficiently small $\epsilon > 0$.
\end{lemma}
\begin{proof}
Take $\alpha,\epsilon_0$ as in Theorem \ref{thm_opentrans} with $f=d.$ Take $\gamma (d)$ such that the $\gamma (d)$-neighborhood of $d$ is contained in
$\mbox{Int}\,\mathcal{T}_\alpha^{\epsilon_0}$. Set $d_i(x)=d(x+i\beta)$ and define $\tilde{A}_i:=\{x: \abs{\tilde{d}_i(x)} < \epsilon \}$ and $\tilde{B}_k^j:=\{x: \frac{1}{2^{k}} \epsilon < \abs{\tilde{d}_j(x)} < \frac{1}{2^{k-1}} \epsilon \}$ for $k \in \mathbb{N}$. 

Then, for $\gamma<\gamma(d)$ sufficiently small and $\epsilon$ sufficiently small determined uniformly for $\vert \vert \tilde{D}-D \vert \vert_\delta <\gamma,$ by Theorem \ref{thm_opentrans}, 
\begin{eqnarray}
\left \vert \int_{\abs{\tilde{d}_i(x)} < \epsilon}  \log \abs{\tilde{d}_j(x)} \ud x \right \vert \leq \sum_{k \in \mathbb{N}} \left \vert \int_{\tilde{A}_i \cap \tilde{B}_k^j} \log \abs{\tilde{d}_j(x)} \ud x  \right \vert +  \nonumber \\
 \left \vert \int_{\tilde{A}_i \setminus \cup_{k \in \mathbb{N}} \tilde{B}_k^j} \log \abs{\tilde{d}_j(x)} \ud x \right \vert \leq C_7 \epsilon^\alpha \abs{\log{\epsilon}} ~\mbox{.}
\end{eqnarray}
\end{proof}

Fix $\epsilon<r(\beta)^{-1}.$ Equation (\ref{eq_ldb10}) together with
Lemma \ref{lem_ldb4} , (\ref{eq_ldb9}) and (\ref{eq_ldb11}) imply, for
uniformly large $n$:
\begin{eqnarray} 
\vert \langle \tilde{w}_n \rangle - L_n(\beta, \tilde{D}) \vert \leq \frac{1}{n} \sum_{i=0}^{n-1} \int_{ \abs{\tilde{d}_i(x)} < \mathrm{e}^{-n^\epsilon}} 
 \log  \left( \dfrac{\mathrm{e}^{-n A}}{\prod_{j=0}^{n-1} \abs{\tilde{d}(x+j \beta)}^\frac{1}{2}    } \right) \ud x \leq \nonumber \\
 A n \mathrm{e}^{-\alpha n^\epsilon} + \frac{1}{2} C_7 n^\epsilon \mathrm{e}^{-\alpha n^\epsilon} < \mathrm{e}^{-c_4 n^\epsilon} ~\mbox{,} ~0<c_4<\alpha ~\mbox{.}
\end{eqnarray}

Thus,
\begin{eqnarray} \label{eq_ldb12}
\mu \{x \in \mathbb{T}: \abs{\tilde{u}_n(x) - L_n(\beta, \tilde{D})} > \kappa \} \leq \mu_L\{x: \abs{\tilde{u}_n(x) - \tilde{w}_n(x)} > \frac{\kappa}{2} - \frac{\mathrm{e}^{-c_4 n^\epsilon}}{2}  \} + \nonumber \\
\mu \{x: \abs{\tilde{w}_n(x) - \langle \tilde{w} \rangle_n(x)} > \frac{\kappa}{2} - \frac{\mathrm{e}^{-c_4 n^\epsilon}}{2}  \} \leq \mu_L\{\tilde{\mathfrak{Y}}_n\} + \mu_L\{x: \abs{\tilde{w}_n(x) - \langle \tilde{w}_n \rangle} > \kappa/3\} \nonumber \\
\end{eqnarray}
for $n$ sufficiently large.

Finally, making use of (\ref{eq_ldb14}) and (\ref{eq_ldb11})  we obtain Theorem \ref{thm_LDB}.
\end{proof}

\section{Concluding remarks} \label{subsec_cont_rem}

In \cite{A} the authors also obtain a continuity statement in the frequency (Theorem 2, \cite{A}). Repeating their proof we obtain an analogous statement here:
To this end we denote the set of numbers satisfying the Diophantine condition (\ref{eq_diophantine}) for a given $r>1$ by $DC(r)$.
\begin{theorem} \label{thm_cont_2}
Both $L^\prime(. , .), L(. , .): DC(r)  \times  \mathcal{C}_\delta^\omega(\mathbb{T},M_2(\mathbb{C})) \to \mathbb{R}$ are jointly continuous.
\end{theorem}

Note that Theorem \ref{thm_cont_2} does {\em{not}} give continuity of the LE under rational approximants of the frequency due to the Diophantine condition imposed. 

If however the $\det D(x)$ is bounded away from zero Theorem \ref{thm_cont_2} can be strengthened using the continuity statement from \cite{C}. For $\delta>$ fixed, define the space
\begin{equation}
\mathcal{B}_\delta^\omega(\mathbb{T},M_2(\mathbb{C})) := \left\{ D \in \mathcal{C}_\delta^\omega(\mathbb{T},M_2(\mathbb{C})) : \mathcal{N}(\det D; \abs{\im(z)} \leq \delta) = 0 \right\} ~\mbox{.}
\end{equation}
By Proposition \ref{prop_semicont} $\mathcal{B}_\delta^\omega(\mathbb{T},M_2(\mathbb{C}))$ is open in $\mathcal{C}_\delta^\omega(\mathbb{T},M_2(\mathbb{C}))$.

\begin{theorem} \label{thm_cont_freq}
If $\beta_0$ is irrational, $L^\prime(. , .), L(. , .): \mathbb{R} \times \mathcal{B}_\delta^\omega(\mathbb{T},M_2(\mathbb{C})) \to \mathbb{R}$ are jointly continuous at $(\beta_0,.)$.
\end{theorem}
\begin{proof}
The statement is known for $L(.,.)$ \cite{C}. To prove continuity for $L^\prime$, let $\beta$ be irrational and fix a $D \in \mathcal{B}_\delta^\omega(\mathbb{T},M_2(\mathbb{C}))$. It suffices to show that if $r_n = \frac{p_n}{q_n}$ is a sequence of {\em{rational}} approximants of $\beta$, $(p_n,q_n)=1$, and $\lVert D_n -D \rVert_\delta \to 0$ we have $L^\prime(r_n, D_n) \to L^\prime(\beta,D)$.

First, notice that for any $n$, $r_n \in \mathbb{Q}$ implies
\begin{equation}
L(r_n, D_n^\prime) = L(r_n, D) - \frac{1}{2 q_n} \sum_{j=0}^{q_n-1} \log \abs{\det D(x+j r_n)} ~\mbox{.}
\end{equation}
Thus, the claimed continuity property of $L^\prime(.,.)$ is reduced to establishing
\begin{equation} \label{eq_cont_freq}
\frac{1}{q_n} \sum_{j=0}^{q_n-1} \log \abs{\det D(x+j r_n)} \to \int_\mathbb{T} \log \abs{\det D(x)} \ud x ~\mbox{,}
\end{equation}
as $n \to \infty$.

Using harmonicity of $\log \det D(x)$ on $\abs \im(z) \leq \delta$,
\begin{equation}
 \frac{1}{q_n} \sum_{j=0}^{q_n-1} \log \abs{\det D(x+j \beta)} \to \int_\mathbb{T} \log \abs{\det D(x)} \ud x ~\mbox{,}
\end{equation}
uniformly on $\mathbb{T}$ as $n\to \infty$.

Hence, (\ref{eq_cont_freq}) follows by successive approximation also making use of Lemma \ref{eq_lemint}:
\begin{equation}
\begin{split}
\left \vert \frac{1}{q_n} \sum_{j=0}^{q_n-1} \right. &\left.  \log \abs{\det D_n(x+j r_n)} - \int_\mathbb{T} \log \abs{\det D_n(x)} \ud x \right \vert \leq \frac{1}{q_n} \sum_{j=0}^{q_n-1}   \big \vert\log \left\vert \det D_n(x+j r_n)\right\vert \\ 
 & - \log \left\vert \det D(x+ j \beta)\right\vert  \big \vert + \left \vert \frac{1}{q_n} \sum_{j=0}^{q_n-1} \log \abs{\det D(x+j \beta)} - \int_\mathbb{T} \log \abs{\det D(x)} \ud x \right\vert \\ & + \left\vert \int_\mathbb{T} \log \abs{\det D(x)} \ud x - \int_\mathbb{T} \log \abs{\det D_n(x)} \ud x  \right \vert ~\mbox{.}
\end{split}
\end{equation}
\end{proof}

\bibliographystyle{amsplain}

\begin{thebibliography}{10}
\bibitem {A} S. Jitomirskaya, D.A. Koslover, M.S. Schulteis \textit{Continuity of the Lyapunov Exponent for analytic quasi-periodic cocycles}, Ergod. Th. \& Dynam. Sys. 29 (2009), 1881 - 1905.
\bibitem {C} J. Bourgain, S. Jitomirskaya, \textit{Continuity of the Lyapunov exponent for quasi-periodic operators with analytic potential}, J. Statist. Phys. 108 (2002), no. 5-6, 12031218.
\bibitem {E} S. Jitomirskaya, D.A. Koslover and M.S. Schulteis, \textit{Localization for a Family of One-dimensional Quasi-periodic Operators of Magnetic Origin}, Ann. Henri Poincar\`{e} 6 (2005), 103 - 124.
\bibitem {G} J. H. Han and D. J. Thouless, H. Hiramoto, M. Kohmoto, \textit{Critical and bicritical properties of Harper's equation with next-nearest neighbor coupling}, Phys. Rev. B 50 (1994), 11365 - 11380.
\bibitem {P} S. Lojasiewicz, \textit{Sur le probl\'eme de la division}, Studia Math. 18 (1959), 87 -136.
\bibitem {S} A. Ya. Khinchin, \textit{Continued fractions}, Dover Publications, New York (1997).
\bibitem {T} P. G. Harper, \textit{Single band motion of conducting electrons in a uniform magnetic field}, Proc. Phys. Soc. A 68 (1955), 874-878.
\bibitem {U} J. M. Luttinger, \textit{The effect of a magnetic field on electrons in a periodic potential}, Phys. Rev. 84 (1951), 814-817.
\bibitem {Y} Y. Katznelson, \textit{An introduction to Harmonic Analysis}, Cambridge University Press, 3rd edition, Cambridge UK (2004).
\bibitem {Z} S. Axler, P. Bourdon and W. Ramey, \textit{Harmonic function theory}, Springer, 2nd edition, New York (2001).
\bibitem {BB} G. P\'olya, \textit{Beitrag zur Verallgemeinerung des Verzerrungssatzes auf mehrfach zusammenhangenden Gebieten}, Sitzungsber. Preuss. Akad. Wiss. Berlin (1928), 228 - 232.
\bibitem {Q} S. Ya. Jitomirskaya, \textit{Metal-insulator transition for the almost Mathieu operator}, Annals of Mathematics 150 (1999), 1159 - 1175.
\bibitem {R} A. Avila and S. Jitomirskaya, \textit{The ten Martini problem}, Annals of Mathematics 170 (2009), 303 - 342.
\bibitem {S} M. Goldsthein, W. Schlag, \textit{H\"older continuity of the integrated density of states for quasi-periodic Schr\"odinger equations and averages of shifts of subharmonic functions}, Annals of Mathematics 154, 155- 203 (2001).
\end{thebibliography}

\end{document}